\documentclass[english]{amsart}

\usepackage{amsmath}
\usepackage{amssymb}
\usepackage{amsthm}
\usepackage{amscd}
\usepackage{babel}
\usepackage[latin1]{inputenc}
\usepackage{graphicx}

\newtheorem{teor}{Theorem}
\newtheorem{cor}{Corollary}
\newtheorem{prop}{Proposition}
\newtheorem{con}{Conjecture}

\newtheorem{defi}{Definition}

\newtheorem*{rem}{Remark}

\renewcommand{\subjclassname}{AMS \textup{2010} Mathematics Subject
Classification\ }

\author{Jos\'{e} Mar\'{i}a Grau, Antonio M. Oller-Marc\'{e}n }

\title{Generalizing Giuga's conjecture}

\begin{document}
\maketitle

\begin{abstract} In 1950 G. Giuga studied the congruence $\sum_{j=1}^{n-1}  j^{n-1} \equiv -1$ (mod $n$) and conjectured that it was only satisfied by prime numbers. In this work we generalize Giuga's ideas considering, for each $k \in \mathbb{N}$, the congruence $\sum_{j=1}^{n-1}  j^{k(n-1)} \equiv -1$ (mod $n$). It particular, it is proved that a pair $(n,k)\in \mathbb{N}^2$ (with composite $n$) satisfies the congruence if and only if $n$ is a Giuga Number and  $ \lambda(n)/\gcd(\lambda(n),n-1)$ divides $k$. In passing, we establish some new characterizations of Giuga Numbers.
\end{abstract}
\subjclassname{11B99, 11A99, 11A07}

\keywords{Keywords: Giuga Numbers, Carmichael function}

\section{Introduction}
In 1950, G. Giuga conjectured (Giuga's conjecture, see \cite{GIU}) that if an integer $n$ satisfies $\displaystyle{\sum_{j=1}^{n-1}  j^{n-1} \equiv -1}$ (mod $n$), then $n$ must be a prime. Moreover, Giuga proved that $n$ is a counterexample to his conjecture if and only if for each prime divisor $p$ of $n$, $(p-1)|(n/p-1)$ and $p|(n/p-1)$. Using this characterization, he proved computationally that any counterexample has at least 1000 digits. Equipped with more computing power, E. Bedocchi (see \cite{BEDO}) later raised this bound to 1700 digits. Improving their method, D. Borwein, J. M. Borwein, P. B. Borwein and R. Girgensohn (see \cite{BOR}) determined that  any counterexample contains at least 3459 distinct primes and so has at least 13887 digits.

The second of the above conditions ($p|(n/p-1)$) motivated the following definition of Giuga Numbers, introduced by Borwein et al. in the paper \cite{BOR}, where these numbers were studied in detail.

\begin{defi} 
A Giuga Number is a composite number $n$ such that $p|(n/p-1)$ for every $p$, prime divisor of $n$.
\end{defi}

There are several characterizations of Giuga Numbers. The most important being the following.

\begin{prop} 
Let $n$ be a composite integer. Then, the following are equivalent:
\begin{itemize}
\item [i)] $n$ is a Giuga Number.
\item [ii)] $\displaystyle{\sum_{p|n} \frac{1}{p}-\prod_{p|n} \frac{1}{p} \in \mathbb{N}}$ (see \cite{GIU}).
\item[iii)] $\displaystyle{\sum_{j=1}^{n-1}  j^{\phi(n)} \equiv -1}$ (mod $n$), where $\phi$ is Euler's totient function (see \cite{BOR}).
\item [iv)] $nB_{\phi(n)} \equiv -1$ (mod $n$), where $B$ is a Bernoulli number (see \cite{AGO}).
\end{itemize}
\end{prop}

It has been conjectured by Paolo P. Lava (see \cite{LAVA}) that Giuga Numbers are the solutions of the differential equation $n'=n+1$, with $n'$ being the arithmetic derivative of $n$.

Up to date only thirteen Giuga Numbers are known (see A007850 in the \emph{On-Line Encyclopedia of Integer Sequences}):
\begin{itemize}
\item With 3 factors: 
$$\textbf{30} = 2 \cdot 3 \cdot 5.$$
\item With 4 factors:
\begin{align*}
\textbf{858}&=2 \cdot 3 \cdot 11 \cdot 13,\\
\textbf{1722}&=2 \cdot 3 \cdot 7 \cdot 41.
\end{align*}
\item With 5 factors:
$$\textbf{66198} = 2 \cdot 3 \cdot 11 \cdot 17 \cdot 59.$$
\item With 6 factors:
\begin{align*}
\textbf{2214408306}&=2 \cdot 3 \cdot 11 \cdot 23 \cdot 31 \cdot 47057,\\
\textbf{24423128562}&=2 \cdot 3 \cdot 7 \cdot 43 \cdot 3041 \cdot 4447.
\end{align*}
\item With 7 factors:
\begin{align*}
\textbf{432749205173838}&=2 \cdot 3  \cdot 7 \cdot 59 \cdot 163 \cdot 1381 \cdot 775807,\\
\textbf{14737133470010574}&=2 \cdot 3 \cdot 7 \cdot 71 \cdot 103 \cdot 67213 \cdot 713863,\\
\textbf{550843391309130318}&=2 \cdot 3 \cdot 7 \cdot 71 \cdot 103 \cdot 61559 \cdot 29133437.
\end{align*}
\item With 8 factors:
\begin{align*}
\textbf{244197000982499715087866346}=\ & 2 \cdot 3 \cdot 11 \cdot 23 \cdot 31 \cdot 47137 \cdot 28282147 \cdot \\ &3892535183,\\
\textbf{554079914617070801288578559178}=\ & 2 \cdot 3 \cdot 11 \cdot 23 \cdot 31 \cdot 47059 \cdot 2259696349 \cdot\\ & 110725121051,\\
\textbf{1910667181420507984555759916338506}=\ & 2 \cdot 3 \cdot 7 \cdot 43\cdots 1831 \cdot 138683 \cdot 2861051 \cdot \\ &1456230512169437.
\end{align*}
\end{itemize}

There are no other Giuga Numbers with less than 8 prime factors. There is another known Giuga Number (found by Frederick Schneider in 2006) which has 10 prime factors, but it is not known if there is any Giuga Number between this and the previous ones. This biggest known Giuga Number is:

\begin{align*}
&\textbf{4200017949707747062038711509670656632404195753751630609228764416}\\ &\textbf{142557211582098432545190323474818}= 2 \cdot 3 \cdot 11 \cdot 23 \cdot 47059 \cdot 2217342227 \cdot\\ &1729101023519 \cdot 8491659218261819498490029296021 \cdot\\
&658254480569119734123541298976556403.
\end{align*}

Observe that all known Giuga Numbers are even. If an odd Giuga Number exists, it must be the product of at least 14 primes. It is not even known if there are infinitely many Giuga Numbers.

The most important result about Giuga's conjectureit is due to Giuga himself (see \cite{GIU} or \cite{BOR}).

\begin{prop} 
A composite integer $n$ is a counterexample to Giuga's conjecture if and only if it is both a Carmichael Number and a Giuga Number.
\end{prop}

On the other hand, Luca, Pomerance and Shparlinski (see \cite{LUCA}) have established the following bound for the counter function of the counterexamples to Giuga's conjecture thus improving a result by Tipu in \cite{TIPU}: 
$$|\{n<X | n \textrm{ is countarexample of Giuga's conjecture }\}| \ll \frac{X^{\frac{1}{2}}}{(\log(X))^2}.$$
Finally we will mention that Bernd C. Kellner \cite{KELL} has stablished that Giuga's conjecture is equivalent to the following conjecture by Agoh.

\begin{con}[Takashi Agoh, 1990] 
Let $B_k$ denote the k-th Bernoulli number. Then, $n$ is a prime if and only if $nB_n \equiv -1 (\textrm{mod. n})$
\end{con}

\section{New characterizations for Giuga Numbers}

The following result will be the core of the paper. It stablishes that in Proposition 1 iii) we can replace Euler's totient function $\phi(n)$ by Carmichael's function $\lambda(n)$ or by any multipli of $\phi(n)$ or $\lambda(n)$.

\begin{prop} 
For every natural numbers $A$, $B$ and $N$ we have that:
$$\sum_{j=1}^{N-1}  j^{A \lambda(N)} \equiv  \sum_{j=1}^{N-1}  j^{B \phi(N)}\ \textrm{(mod $N$)}.$$
\end{prop}
\begin{proof}
Put $N=2^ap_1^{r_1}\cdots p_s^{r_s}$ with $p_i$ distinct odd primes. Choose $i\in \{1,\dots, s\}$. We have that:
$$\sum_{j=1}^{N-1} j^{A\lambda(N)} \equiv \frac{N}{p_i^{r_i}}\sum_{j=1}^{p_i^{r_i}-1} j^{A\lambda(N)}\ \textrm{(mod $p_i^{r_i}$)}.$$
$$\sum_{j=1}^{N-1} j^{B\phi(N)} \equiv \frac{N}{p_i^{r_i}}\sum_{j=1}^{p_i^{r_i}-1} j^{B\phi(N)}\ \textrm{(mod $p_i^{r_i}$)}.$$
Now, since both $A\lambda(N),\ B\phi(N)\geq r_i$, we get:
$$\sum_{j=1}^{p_i^{r_i}-1} j^{A\lambda(N)}=\sum_{\substack{1\leq j\leq p_i^{r_i}-1\\ (p_i,j)=1}} j^{A\lambda(N)}+\sum_{\substack{1\leq j\leq p_i^{r_i}-1\\ p_i|j}} j^{A\lambda(N)}\equiv \phi(p_i^{r_i})+0\ \textrm{(mod ($p_i^{r_i}$)}.$$
$$\sum_{j=1}^{p_i^{r_i}-1} j^{B\phi(N)}=\sum_{\substack{1\leq j\leq p_i^{r_i}-1\\ (p_i,j)=1}} j^{B\phi(N)}+\sum_{\substack{1\leq j\leq p_i^{r_i}-1\\ p_i|j}} j^{A\lambda(N)}\equiv \phi(p_i^{r_i})+0\ \textrm{(mod ($p_i^{r_i}$)}.$$
Consequently:
$$\sum_{j=1}^{N-1} j^{A\lambda(N)}\equiv\sum_{j=1}^{N-1} j^{B\phi(N)}\ \textrm{(mod $p_i^{r_i}$)}\ \textrm{for every $i=1,\dots,s$}.$$
Clearly if $N$ is odd the proof is complete. If $n$ is even we have that:
$$\sum_{j=1}^{N-1}j^{A\lambda(N)}\equiv\frac{N}{2^a}\sum_{j=1}^{2^a-1} j^{A\lambda(N)}\equiv\frac{N}{2^a}\left(\sum_{\substack{1\leq j\leq 2^a-1\\ \textrm{$j$ even}}} j^{A\lambda(N)} + 2^{a-1}\right)\ \textrm{(mod $2^a$)}.$$
$$\sum_{j=1}^{N-1}\equiv\frac{N}{2^a}\left(\sum_{\substack{1\leq j\leq 2^a-1\\ \textrm{$j$ even}}} j^{B\phi(N)} + 2^{a-1}\right)\ \textrm{(mod $2^a$)}.$$
Now, if $a=1,2$ or $3$ it can be easily verified that:
$$\sum_{\substack{1\leq j\leq 2^a-1\\ \textrm{$j$ even}}} j^{A\lambda(N)}\equiv\sum_{\substack{1\leq j\leq 2^a-1\\ \textrm{$j$ even}}} j^{B\phi(N)}\ \textrm{(mod $2^a$)}.$$
On the other hand, if $a\geq 4$ we have that $\phi(N)\geq\lambda(N)\geq a$ and, consequently $j^{A\lambda(N)}\equiv j^{B\phi(N)}\equiv 0$ (mod $2^a$) for every $1\leq j\leq 2^{a-1}$ even. Thus:
$$\sum_{j=1}^{N-1} j^{A\lambda(N)}\equiv\sum_{j=1}^{N-1} j^{B\phi(N)}\ \textrm{(mod $2^a$)}$$
and the result follows.
\end{proof}

The proposition above allows us to introduce some new characterizations of Giuga Numbers. Recall that a composite integer $n$ is said to be a Giuga Number if and only if $\displaystyle{\sum_{j=1}^{n-1} j^{\phi(n)}\equiv -1}$ (mod $n$).

\begin{cor} 
Let $n$ be any composite integer. Then the following are equivalent:
\begin{itemize}
\item[i)] $n$ is a Giuga Number.
\item[ii)] $\displaystyle{\sum_{j=1}^{n-1}  j^{\lambda(n)}} \equiv -1$ (mod $n$).
\item[iii)] For every positive integer $K$, $\displaystyle{\sum_{j=1}^{n-1}  j^{K\lambda(n)}} \equiv -1$ (mod $n$). 
\item [iv)] There exists a positive integer $K$ such that $\displaystyle{\sum_{j=1}^{n-1}  j^{K\lambda(n)}} \equiv -1$ (mod $n$).
\item[v)] There exists a positive integer $K$ such that $\displaystyle{\sum_{j=1}^{n-1}  j^{K\phi(n)} \equiv -1}$ (mod $n$).
\item[vi)] For every positive integer $K$, $\displaystyle{\sum_{j=1}^{n-1}  j^{K\phi(n)} \equiv -1}$ (mod $n$).
\end{itemize}
\end{cor}

\begin{proof} 
It follows readily from the previous proposition and from Proposition 1 iii).
\end{proof}

Proposition 1 also allows us to show in a different and novel way that an integer which is both a Carmichael Number and a Giuga Numberis a counterexample to Giuga's conjecture.

\begin{cor} 
If an integer $n$ is both a Carmichael Number and a Giuga Number, then:
$$\sum_{j=1}^{n-1}  j^{n-1} \equiv -1\ \textrm{(mod $n$)}.$$
\end{cor}

\begin{proof} 
If $n$ is a Carmichael Number, then $\lambda(n) | (n-1)$. If we put $k=\frac{n-1}{\lambda(n)}$, then we have:
$$S:=\sum_{j=1}^{n-1}  j^{n-1}=\sum_{j=1}^{n-1}  j^{k \lambda(n)}.$$
If, in addition, $n$ is a Giuga Number we can apply Proposition 3 with $A=k$ and $B=1$ to get $S\equiv  -1$ (mod $n$) as claimed.
\end{proof}

Although the previous result is not new, it is interesting to observe that the proof avoids the use of Korselt's criterion which is replaced by Carmichael's criterion, i.e., $\lambda(n)|(n-1)$.

The following result, which is also a consequence of Proposition 3, will allow us to generalize Giuga's ideas by considering the congruence $\displaystyle{\sum_{j=1}^{N-1}  j^{k(n-1)} \equiv  -1}$ (mod $n$) for each positive integer $k$.

\begin{cor} 
If an integer $n$ is a counterexample to Giuga's conjecture, then:
$$\sum_{j=1}^{n-1}  j^{k(n-1)} \equiv  -1\ \textrm{(mod $n$) for every positive integer $k$}.$$
\end{cor}
\begin{proof} 
If $n$ is a counterexample to Giuga's conjecture, then it is both a Carmichael and a Giuga Number. Being a Carmichael Number, we have that $\lambda(n) | (n-1)$ so if $\frac{k(n-1)}{\lambda(n)}=\widetilde{k}\in\mathbb{N}$ we get:
$$S:=\sum_{j=1}^{n-1}  j^{k(n-1)}=\sum_{j=1}^{n-1}  j^{k \lambda(n)\frac{(n-1)}{\lambda(n)}}=\sum_{j=1}^{n-1}  j^{\widetilde{k} \lambda(n)},$$
and, since $n$ is a Giuga Number it is enough to apply Corollary 1 and Proposition 3 to get $S\equiv  -1$ (mod $n$).
\end{proof}

\section{Generalizing Giuga's conjecture }

In this section we generalize Giuga's ideas in the following way: Do there exist integers $k$ and $n$ such that the congruence $\displaystyle{\sum_{j=1}^{n-1}  j^{k(n-1)} \equiv  -1}$ (mod $n$) is satisfied by some composite integer $n$? 

 In what follows we will denote the set of Giuga Numbers by $\mathcal{G}$ and the set of Carmichael Numbers by $\mathcal{C}$. Moreover, for every positive integer $k$ let us define the following sets:
$$\mathcal{G}_k:=\left\{n\in\mathbb{N}  \left|\textrm{ $n$ is composite, } \displaystyle{\sum_{j=1}^{n-1}  j^{k(n-1)} \equiv  -1}\right.\ \textrm{(mod $n$)}\right\},$$
$$\mathcal{K}_n:=\left\{k \in \mathbb{N} \left|\ \sum_{j=1}^{n-1}  j^{k(n-1)} \equiv  -1\ \right. \textrm{(mod $n$)}\ \right\}.$$

\begin{rem}
With the previous notation, Giuga's conjecture is equivalent to the statement $\mathcal{G}_1=\emptyset$. Also observe that, for every positive integer $k$:
$$ k \in \mathcal{K}_n\ \textrm{if and only if}\ n \in \mathcal{G}_k.$$
\end{rem}

\begin{teor} 
Let $n$ be a composite integer. Then the following are equivalent:
\begin{itemize}
\item[i)] $ n \in \mathcal{G}_k $.
\item[ii)] $n \in \mathcal{G}$  and $\displaystyle{\frac{\lambda(n)}{\gcd(\lambda(n),n-1)}}$ divides $k$.
\end{itemize}
\end{teor}
\begin{proof} 
Assume that $n\in\mathcal{G}$ and that $\displaystyle{\frac{\lambda(n)}{\gcd(\lambda(n),n-1)}}$ divides $k$. Then we have:
$$\sum_{j=1}^{n-1}  j^{k(n-1)} = \sum_{j=1}^{n-1}  j^{\frac{\lambda(n)}{gcd(\lambda(n),n-1)} k' (n-1)}=\sum_{j=1}^{n-1}  j^{k''\lambda(n)}\equiv -1$$
due to Corollary 1.

Conversely, assume that $n\in\mathcal{G}_k$. This means that $\displaystyle{\sum_{j=1}^{n-1}j^{k(n-1)}\equiv -1}$ (mod $n$). As a consequence (see \cite{WONG}[Theorem 2.3]) we have that $p-1$ divides $k\left(n/p-1\right)$ and that $p$ divides $n/p-1$ for every $p$, prime divisor of $n$. and, moreover, $n$ is square-free. Since $n$ is square-free we have that $\lambda(n)=\textrm{lcm}\{p-1\ |\ \textrm{$p$ odd prime dividing $n$}\}$. Thus, $\lambda(n)$ divides $k(n-1)$ and, consequently, $\displaystyle{\frac{\lambda(n)}{\gcd(\lambda(n),n-1))}}$ divides $k$. It is enough to apply Proposition 3 with $B=1$ and $A=\frac{n-1}{\gcd(\lambda(n),n-1)}$ to finish the proof.
\end{proof}

If we put $k=1$ in the theorem above, we obtain again Corollary 2 and also its reciprocal thus completing the characterization of counterexamples to Giuga's conjecture without the use of Korselt's criterion.

\begin{cor} 
An integer $n$ is a counterexample to Giuga's conjecture if and only if it is both a Carmichael and a Giuga Number. In other words: $$\mathcal{G}_1=\mathcal{G}\cap \mathcal{C}.$$
\end{cor}
\begin{proof}
Let $n \in \mathcal{G}_1$. Then, by the previous theorem $n \in \mathcal{G} $ and $\gcd(\lambda(n),n-1)=\lambda(n)$; i.e., $n \in \mathcal{C}$.

Now, let $n \in \mathcal{C}\cap\mathcal{G}$. Then $\frac{\lambda(n)}{\gcd(\lambda(n),n-1)}=1$ and also the previous theorem implies that $n \in \mathcal{G}_1$. This completes the proof.
\end{proof}

By Theorem 1 we can find values of $k$ such that $\mathcal{G}_k$ is non-empty. To do so, we evaluate $\frac{\lambda(n)}{gcd(\lambda(n),n-1)}$ for every known Giuga Number. Thus, we will have thirteen values of $k$ for which $\mathcal{G}_{tk}$ is known to be nonempty for any $t$. In the case $t=1$, they are:

$$\mathcal{G}_4=\{30, ...\};$$
$$\mathcal{G}_{60}=\{30, 858, ...\};$$
$$\mathcal{G}_{120}=\{30, 858, 1772, ...\};$$
$$\mathcal{G}_{2320}=\{30, 66198,...\};$$
$$\mathcal{G}_{1552848}=\{30,2214408306,...\};$$
$$\mathcal{G}_{10080}=\{30,858, 24423128562,...\};$$
$$\mathcal{G}_{139714902540}=\{30,858,432749205173838,...\};$$
$$\mathcal{G}_{93294624780}=\{30,858,14737133470010574,...\};$$
$$\mathcal{G}_{228657996794220}=\{30,858,550843391309130318,...\};$$
$$\mathcal{G}_{4756736241732916394976}=\{30,244197000982499715087866346,...\};$$
$$\mathcal{G}_{20024071474861042488900}=\{30,554079914617070801288578559178,...\};$$
$$\mathcal{G}_{2176937111336664570375832140}=\{30,858,1910667181420507984555759916338506,...\};$$
$$\mathcal{G}_{15366743578393906356665002406454800354974137359272445859047945613961394951904884493965220}$$ $$=\{30,858,42000179497077470620387115096706566324041957537516306092287644$$ $$\hspace{-5cm} 16142557211582098432545190323474818,...\}.$$

Given any positive integer $n$, Theorem 1 gives a complete description of the set $\mathcal{K}_n$ as can be seen in the following corollary.

\begin{cor} 
Let $n$ be any positive integer. Then:
$$\mathcal{K}_n =\begin{cases}  \left\{\frac{t\lambda(n)}{gcd(\lambda(n),n-1)}\ |\ t\in \mathbb{N} \right\}, & \textrm{if $n \in \mathcal{G} $;}\\ \emptyset, & \textrm{otherwise.}\\ \end{cases}$$
\end{cor}

\begin{rem}
Observe that $\mathcal{K}_n=\mathbb{N}$ if and only if $n$ is a counterexample to Giuga's conjecture.
\end{rem}

\section{Reflecting about Giuga's conjecture }

In recent work by W. D. Banks, C. W. Nevans, and C. Pomerance  \cite{POM}, the following bounds were given:

\begin{teor} For any fixed $ \varepsilon> 0$, $\beta=0.3322408$ and all sufficiently large X, we have
$$| \{n<X | n \in \mathcal{C} \} | \geq X^{\beta-\varepsilon} \textrm{ (due to G. Harman \cite{HAR} )}$$
$$| \{n<X | n \in \mathcal{C} \backslash \mathcal{G}_1 \} | \geq X^{\beta-\varepsilon} $$
\end{teor}

The authors of the aforementioned paper consider the above bounds to be consistent with Giuga's conjecture. We believe, however, that the same consideration could be made with respect to conjectures that are actually false, such as  $\mathcal{G}_4=\emptyset$ or $\mathcal{G}_{1552848}=\emptyset$. What is clearly consistent is the assumption that the elements of $\mathcal{G}_{k}$, regardless of the value of $k$, are extremely ``scarce''.

Furthermore, the authors of the present paper, in view of the generalization presented here, are convinced that Giuga's conjecture is not based on any sound logical-mathematical consideration and that its strength rests on the extreme rarity of Giuga Numbers, combined with the computational difficulties that the search for new numbers belonging to this family entail, as well as with the null asymptotic density of Carmichael numbers. In fact, if we may be forgiven the joke, we might conjecture -without any fear of our conjecture being refuted in many years- that $\mathcal{G}_2=\mathcal{G}_3=\mathcal{G}_5=\emptyset$ or, to be even more daring, that  $\mathcal{G}_p=\emptyset$ for all prime $p$. Of course, Giuga's conjecture has the honour of being the strongest of all of these conjectures. In fact, in virtue of Corollary 3, should it be refuted, all the others would fall with it.

\bibliography{./refgiugabis}

\begin{thebibliography}{10}

\bibitem{AGO}
Takashi Agoh.
\newblock On {G}iuga's conjecture.
\newblock {\em Manuscripta Math.}, 87(4):501--510, 1995.

\bibitem{LAVA}
Giorgio Balzarotti and Paolo~P. Lava.
\newblock {\em 103 curiosità matematiche}.
\newblock Hoepli, 2010.

\bibitem{POM}
William~D. Banks, C.~Wesley Nevans, and Carl Pomerance.
\newblock A remark on {G}iuga's conjecture and {L}ehmer's totient problem.
\newblock {\em Albanian J. Math.}, 3(2):81--85, 2009.

\bibitem{BEDO}
Edmondo Bedocchi.
\newblock Note on a conjecture about prime numbers.
\newblock {\em Riv. Mat. Univ. Parma (4)}, 11:229--236, 1985.

\bibitem{BOR}
D.~Borwein, J.~M. Borwein, P.~B. Borwein, and R.~Girgensohn.
\newblock Giuga's conjecture on primality.
\newblock {\em Amer. Math. Monthly}, 103(1):40--50, 1996.

\bibitem{GIU}
Giuseppe Giuga.
\newblock Su una presumibile propriet\'a caratteristica dei numeri primi.
\newblock {\em Ist. Lombardo Sci. Lett. Rend. Cl. Sci. Mat. Nat. (3)},
  14(83):511--528, 1950.

\bibitem{HAR}
Glyn Harman.
\newblock On the number of {C}armichael numbers up to {$x$}.
\newblock {\em Bull. London Math. Soc.}, 37(5):641--650, 2005.

\bibitem{KELL}
Bernd~C. Kellner.
\newblock The equivalence of giuga's and agoh's conjectures.
\newblock {\em arXiv:math/0409259v1 [math.NT]}.

\bibitem{LUCA}
Florian Luca, Carl Pomerance, and Igor Shparlinski.
\newblock On {G}iuga numbers.
\newblock {\em Int. J. Mod. Math.}, 4(1):13--18, 2009.

\bibitem{TIPU}
Vicentiu Tipu.
\newblock A note on {G}iuga's conjecture.
\newblock {\em Canad. Math. Bull.}, 50(1):158--160, 2007.

\bibitem{WONG}
Erick Wong.
\newblock {\em Computations on Normal Families of Primes}.
\newblock M.Sc. Thesis. Simon Fraser University, 1997.

\end{thebibliography}
\bibliographystyle{plain}

\end{document}